\documentclass{amsart}
\usepackage{amssymb,amsmath}
\usepackage{amsfonts,amsthm}
\usepackage{xcolor}
\usepackage{amssymb}
\usepackage{amsfonts}
\usepackage[english]{babel}
\usepackage[utf8]{inputenc}
 
\usepackage[all]{xy}
\usepackage{amsfonts,amsmath,verbatim,amssymb,amscd,latexsym}
\usepackage{amsthm,calc,enumerate, amscd}
\usepackage{graphicx}
\usepackage[T1]{fontenc}
\usepackage{pdfpages}
\usepackage{mathtools} 
\usepackage{textcomp} 
\newtheorem{theorem}{Theorem}[section]
\newtheorem{lemma}[theorem]{Lemma}

\newtheorem*{thm1.1}{\bf Theorem 1.1}
\newtheorem*{thm1.2}{\bf Theorem 1.2 (Unconditional)}
\newtheorem*{lem5.3}{\bf Lemma 5.3}
\newtheorem*{thmv}{ Theorem:  Blomer (\cite{VB})} 
\newtheorem*{thml}{ Theorem:  Lin (\cite{ LIN})} 

\theoremstyle{definition}

\newtheorem{corollary}{Corollary}

\theoremstyle{remark}
\newtheorem{remark}[]{{ \bf Remark:}}

\numberwithin{equation}{section}
\begin{document}
\title{On double shifted convolution sum of  $SL(2, \mathbb{Z})$ Hecke eigen forms } 
\author{Saurabh Kumar Singh}

\address{ Stat-Math Unit,
Indian Statistical Institute, 
203 BT Road,  Kolkata-700108, INDIA.}

\email{skumar.bhu12@gmail.com}

\date{\today}


\maketitle

\begin{abstract}
Let $\lambda_i (n)$ $i= 1, 2, 3$  denote the normalised  Fourier coefficients of  holomorphic eigenform or   Maass cusp form. In this paper we shall consider the sum: 

\[
 S:=  \frac{1}{H}\sum_{h\leq H} V\left( \frac{h}{H}\right)\sum_{n\leq  N}  \lambda_1 (n)  \lambda_2 (n+h) \lambda_3 (n+ 2h)W\left( \frac{n}{N} \right),
\] 
\noindent
 where $V$ and $W$ are  smooth bump functions, supported on  $[1, 2]$.  We shall prove a  nontrivial upper bound, under the assumption that $H\geq N^{1/2+ \epsilon}$.
 
\end{abstract}

\section{ I\lowercase{ntroduction} }

\noindent
 The study of the shifted convolution  sum $D_k(N, h) := \sum_{ N<n \leq 2N} d_k(n) d_k(n+h)$ of generalised divisor function $ d_k(n)$ is a central problem in number theory, where $ d_k(n)$ is defined to be the Dirichlet coefficient of $ \zeta^k(s) $ in the half plane $ \Re (s)> 1 $. The sum $D_k(N, h)$ comes naturally in the computation of $ 2k$-th power moment of Riemann Zeta function, which is defined as
 \[
  I_k(T):= \int_1^T |\zeta(s)|^{2k} ds. 
 \]
 
 \noindent
The gereral additive divisor problem consists of estimation of the quantity $ \Delta_k (x, h)$, which is given by the equation

\begin{equation} \label{generaladditive}
\sum_{ n \leq x} d_k(n) d_k(n+h) = x P_{2k-2} (\log x; h) +  \Delta_k (x, h), 
\end{equation} 
where $k\geq 2$ is a fixed integer, $ P_{2k-2} (\log x; h) $ is a suitable polynomial of degree $ 2k-2$ in $\log x$ whose coefficients depend on $k$ and $h$,  and $  \Delta_k (x, h)$ is supposed to be the error term.  From the above analogy, it is expected that  $D_k(N, h)$ should be asymptotic to $ c_{k, h} N \log^{2k-2} N$, for some constant  $ c_{k, h} >0 $, uniformaly for $h$ in some range. 
Even for a fixed h, this has only been proved for $k\leq2$, and no proof exists for $k\geq3$.

 \noindent
The behaviour of $D_k(N, h)$ for $k=1$ is very simple. For $k=2$,  Ingham (\cite{IN}) first established the asymptotic formula that

\[
D_2 (N, h) \sim \frac{6}{\pi^2} \sigma_{-1} (h) N \log^2 N,
\] for any $h \in \mathbb{N}$, where $ \sigma_{-1} (h) = \sum_{ j \mid h} j^{-1} $. Later, T. Estermann  established the asymptotic expansion (see \cite{ET})

\[
\sum_{n\leq x} d(n) d(n+h) = x P_h(\log x) + O \left( x^{11/12} \log^x \right), 
\]  where $ P_h (x)$ is a quadratic polynomial with leading coefficient $ \frac{6}{\pi^2} \sigma_{-1} (h)$. 
Many authors have since revisited this problem. The best known results for the error term are due to Duke, Friedlander and Iwaniec \cite{DFI} and Meurman \cite{MERU}.

\noindent

For $k=3$, several authors have studied the property of $ D_3 (N, h)$.  Some recent results are given in  \cite{BB} and \cite{IV}.  Averaging over shift $ h\leq H$, Baier, Browning, Marasingha and Zhao ( \cite{BB} ) have given an asymptotic formula for $G(N, H):= \sum_{h\leq H} D(N, H) $, provided $ N^{1/6 + \epsilon } \leq H \leq N^{1- \epsilon}$.  Precisely, they proved that 

\[
\sum_{h\leq H} \Delta(N, h) \ll ( H^2 + H^{1/2} N^{13/12} ) N^\epsilon,
\] where $1\leq H \leq N $, and $ \sum_{h\leq H} \Delta(N, h) $ is given by equation \eqref{generaladditive}. For $k\geq 3$ the behaviour of $ \Delta(N, h) $ has been studied  by Ivi{\' c} and Wu (\cite{IV}). They proved that, for $k\geq 3$ we have

$$ \sum_{h\leq H} \Delta_k(N, h) \ll ( H^2 +  N^{1 +\beta_k} ) N^\epsilon  \hspace{1cm} (1\leq H\leq N), $$
where $ \beta_k $ is defined by 
$$ \beta_k := \inf \left\lbrace b_k : \int_1^X | \Delta(x)|^2 dx \ll X^{1+2b_k} \right\rbrace,$$ where 
$ \Delta(x) $ is defined by equation 
$$ D_k(x) := \sum_{n \leq x} d_k(n)= x p_{k-1} (\log x) + \Delta(x).$$

\noindent
In this paper we will study the behaviour of double shifted convolution sum of the coefficients of  holomorphic cusp forms of weight $k$, or  Maass eigenforms  on the full modular group $ SL(2, \mathbb{Z})$ (for details see for example \cite{HI} and \cite{IK} ).  We shall denote the space of such form by $ \mathfrak{S} (\mathbb{Z})$. Let  $f: \mathbb{H} \rightarrow \mathbb{C}$ be such that $ f\in \mathfrak{S} (\mathbb{Z}) $.  Since $f(z+1)= f(z)$, $f$ admits a Fourier expansion at infinity and we denote its normalised  $n$th Fourier coefficient by $\lambda_f (n) $.

\noindent
Double shifted convolution sum for Fourier coefficients is  given by:

\begin{equation} \label{doubleshifted}
S(N, h):= \sum_{n\leq N} \lambda_1(n) \lambda_2(n+h) \lambda_3(n+2h), 
\end{equation}
\noindent
where $\lambda_i \in  \mathfrak{S} (\mathbb{Z})$ for  $ i = 1, 2, 3$.  Summing over all  shifts $h\leq H$ we shall  prove the following theorem:  

\begin{theorem} \label{maintheorem}
 With  $ S(N, h)$ defined as above, assume that $ H \gg N^{1/2+ \epsilon}$, where $\delta>0$. Then there exists a positive constant $ \delta = \delta(\epsilon)$ such that
\[
\frac{1}{H} \sum_{h\leq H}  V\left( \frac{h}{H} \right)\sum_{n\leq N} \lambda_1(n) \lambda_2(n+h) \lambda_3(n+2h) W\left( \frac{n}{N}\right) \ll N^{1- \delta},
\] 
where  $V$ and $W$ are  smooth bump functions, supported on the interval  $[1, 2]$.
 
\end{theorem}

One of the motivation for studying the double shifted convolution sum for Fourier coefficients is the corresponding problem for the divisior function. Successful analysis of the sum
\begin{equation} \label{divisorshift}
T_h (X)= \sum_{n\leq X} d(n-h) d(n) d(n+h) 
\end{equation}
has not been completed yet, even for a single positive integer $h$. It is conjectured that
 $ T_h (X) \sim c_h X \log^3 X$, for  a suitable constant $c_h>0$. A 
heuristic analysis based on the underlying Diophantine equations suggests that one should take 
$$ c_h= \frac{11}{8} f(h) \prod_p \left( 1- \frac{1}{p}\right)^2 \left( 1+ \frac{2}{p}\right),$$
where $f$ is a suitable multiplicative function. It has been proved by  Browning  that ( see \cite[Theorem 2]{TD3})

\begin{equation} \label{browningtheorem}
\sum_{h\leq H} \left( T_h (X) - c_h X \log^3 X \right) = o(H  X \log^3 X), 
\end{equation} 
provided $H\geq X^{3/4 + \epsilon} (\epsilon >0).$ 

Recently, Valentine Blomer improved the range of $H$ substantially to $X^{1/3+ \epsilon}$ by using spectral theory of automorphic forms. His method is flexible enough to adopt for more general correlation sum. He proves the following theorem

\begin{thmv} 
  Let $W$ be a smooth function with compact support in $[1, 2]$ with Mellin transform $ \widehat{W}.$ 
 Let $1\leq H \leq N/3 $ and let $ k \geq 2$ be an integer. Let $a_n$,  $N \leq n \leq2N$, be any sequence of complex numbers and let $r_d (n)$ denotes the Ramanujan sum. Then
 \begin{align*}
 \sum_{h\leq H} W\left( \frac{h}{H}\right) \sum_{N\leq n \leq 2N} a(n) \tau(n-h) \tau(n+h) & = H \widehat{W} (1)  \sum_{N\leq n \leq 2N} a(n) \sum_{ d} \frac{r_d(2n)}{d^2} (\log n + 2\gamma  - 2 \log d)^2  \\
 &  +O\left( N^\epsilon \left(\frac{H^2}{N^{1/2}} + HN^{1/4} + (HN)^{1/2} + \frac{N}{H^{1/2}} \right) ||a||_2 \right), 
 \end{align*}
 where the $O$-constant depends on (the Sobolev norms of ) $W$ and $\epsilon$, and $ ||a||$ is  $\ell^2$-norm. 
\end{thmv} Using this, Blomer obtains the following corollary
 
 \begin{corollary}
    Let $W$ be a smooth function with compact support in $[1, 2]$ with Mellin transform $ \widehat{W}.$ 
 Let $1\leq H \leq N/3 $ and let $ k \geq 2$ be an integer. Then
 
\begin{align} \label{blomer}
 \sum_{h\leq H} W\left( \frac{h}{H}\right) \sum_{N\leq n \leq 2N} \tau_k (n) \tau(n-h) \tau(n+h) & =  \widehat{W} (1)  HN Q_{k+1} (\log N)   \notag\\
 &  \hspace{5pt} +O\left( N^\epsilon \left(H^2  + N H^{1/2} + N^{3/2} H^{-1/2} + HN^{1- \frac{1}{k+2}} \right) \right), 
  \end{align}
  where $\tau_k$ denotes the $k$-fold divisor function, $Q_k+1$ is a polynomial (depending only on $k$) of degree $k + 1$ and leading constant,
  
  \[
  \frac{1}{(k-1) !} \prod_p  \left( 1 +\frac{1}{p} \right) \left(  1- \frac{1}{p^{\delta_p =2}} \left( 1 -\frac{1}{p+1}  \right)^k \right), 
  \]
  and the implied constant in the error term depends on (the Sobolev norms of ) $ W$, $k$ and $\epsilon$.
 \end{corollary}
The above asymptotic formula is non trivial (in fact with a power saving error term) for 
$$ N^{1/3+ \epsilon} \leq H \leq N^{1-\epsilon},$$ 
and it is independent of $k$.  In the case of  $k =2$, this  improves Browning’s result substantially. Here, the lower bound for $H$ is coming from the third error term of equation \eqref{blomer},  which is the limit of the automorphic forms machinery.  It is worth noting that the divisor function can be viewed as Fourier coefficients of Eisenstein series. Blomer remarked that  using  Jutila's circle method, one can prove analogous result for Fourier coefficients of cusp form  in place of divisor function. Using this idea ,  Yongxiao Lin \cite{LIN} proved the following theorem:

\begin{thml}
Let $1\leq H \leq X/3$. Let $W$ be a smooth function with compact support in $[1, 2]$, and $a_n$,
$X\leq n \leq 2X$, be any sequence of complex numbers. Let $\lambda_1(n)$, $\lambda_2(n)$ be Hecke eigenvalues of holomorphic Hecke eigencuspforms of weight $\kappa_1$, $\kappa_2$ for $SL(2, \mathbb{Z})$ respectively. Then

\begin{align*}
\sum_{h\leq H} W\left( \frac{h}{H}\right) \sum_{N\leq n \leq 2N} a(n) \lambda_1(n-h) \lambda_2(n+h) \ll
 N^\epsilon  \frac{N}{H}\left( (HN)^{1/2} + \frac{N}{H^{1/2}} \right) ||a||_2. 
\end{align*}

\end{thml}
The error term coming here is comparable to the third and fourth expressions in $O$-term of Blomer's theorem, and both comes from the spectral theory of automorphic forms. As a consequence, Y. Lin proved the following corollary:

\begin{corollary}
Let $1\leq H \leq X/3$. Let $W$ be a smooth function with compact support in $[1, 2]$. Let $\lambda_1(n)$, $\lambda_2(n)$ and  $\lambda_2(n)$ be Hecke eigenvalues of holomorphic Hecke eigencuspforms of weight $\kappa_1$, $\kappa_2$  and $\kappa_2$ for $SL(2, \mathbb{Z})$ respectively. Then

\begin{align*}
\sum_{h\leq H} W\left( \frac{h}{H}\right) \sum_{N\leq n \leq 2N}  \lambda_1(n-h) \lambda_2(n) \lambda_3(n+h) \ll  N^\epsilon \min \left( NH, \frac{N^2}{H^{1/2}} \right).
\end{align*} 
\end{corollary} Note that the result is non-trivial only for $H \geq N^{2/3 + \epsilon}$, for any $\epsilon >0.$

We use the circle method of Heath brown (see \cite{HB}) twice to improve the range of $H$ and the same method also works in the case of Maass forms.  By  using the Voronoi summation formula for the eigenforms on congruence subgroups,  given in appendix $A.4$ of \cite{KMV}, it is possible to prove this theorem for eigenform on the congruence subgroups of $SL(2, \mathbb{Z})$. The calculations will be similar. 

%

\section{\bf Notations and Preliminaries} 
We shall first recall some basic facts about $SL(2, \mathbb{Z})$ automorphic forms. Our requirement is minimal. In fact, the Voronoi summation formula and cancellation in additive twist (see equation \eqref{sl2}) is all that we will be using. Let $f(z)$ $ (z= x+ iy, y>0)$ be a primitive holomorphic Hecke eigenform of integral weight $k(> 2)$ on the full modular group $ SL(2, \mathbb{Z})$. The normalised Fourier expansion of $f$  at cusp $\infty$ is given by 
 $$ f(z)= \sum_{n=1}^\infty \lambda_f (n) n^{(k-1)/2} e(nz) \hspace{1cm} ( \lambda_f(1) =1),$$ 
 where $e(z) = e^{2\pi z}$ and $ \lambda_f (1) = 1$. From Ramanujan-Peterson conjecture, which has been proved by Delign we have $ \lambda_f (n) \leq d(n)$ for every positive integer $n$. Analogously, let $f(z)$ be a primitive Maass cusp form on the group $SL(2, \mathbb{Z})$ with Laplacian eigenvalue $ \frac{1}{4} + \nu^2$.   Then the normalised Fourier expansion of $f$  at cusp $\infty$ is given by 
\[   
\sqrt{y} \sum_{n \neq 0 } \lambda_f(n) K_{i\nu}  (2\pi |n| y ) e(n x), 
\] where $ K_{i\nu}$ denotes the $K$-Bessel function and $ \lambda_f(1) =1$. It follows from the Rankin-Selberg theory that the Fourier coefficients $\lambda_f (n)  $\textquotesingle s are bounded on average, namely

\begin{equation}
\sum_{n\leq X} |\lambda_f (n)|^2 = C_f X + O \left( x^{3/5} \right), 
\end{equation} for some constant $C_f > 0$. Ramanujan-Petersson conjecture predicts that $\lambda_f(n) \ll n^\epsilon$. This has been proved by Delign in case of holomorphic cusp forms, where he proves that $\lambda_f (n) \leq d(n)$. In case of Maass cusp form the best knonwn result is  $\lambda_f (n) \ll n^{7/64 + \epsilon}$, proved by Kim and Sarnak( see \cite{KS}).  On the other hand,  one knows that the Fourier coefficients oscillate quite substantially. For any $ X>1$ and any $\alpha \in \mathbb{R}$, we have 
\begin{equation} \label{sl2}
\sum_{n\leq X} \lambda_{f} (n) e(\alpha n) \ll_{f} X^{\frac{1}{2}} \log (2X), \ \ \  \ \textrm{and} \ \ \ \ 
  \sum_{n\leq X} \lambda_{f} (n) \ll_{f, \epsilon} X^{\frac{1}{3} + \epsilon}. 
\end{equation}
where the implied constant depends only on $f$,  and not on $\alpha$ ( see for example \cite[Page 71, Theorem $5.3$]{HI} and \cite{FI}).  

\noindent
{\bf Notations :} In our work  $\tau (n) $ or $d(n)$ denotes the divisor function. $\epsilon$, $\delta$ denote  small positive constants, which may be different on different occurrence. All of the implied constant in this paper depends on these parameters.
\noindent
%
%
%

\section{\bf Some Lemmas }
\begin{lemma} \label{deltan}
 We define
 \begin{equation} \label{del}
  \delta (n)= 
  \begin{cases}
   1 \ \ \ \ \ \rm{if}\ \ \ \ \  n= 0, \\ 
   0 \ \ \ \ \ \rm{if} \ \ \ \ \ n \neq 0.
  \end{cases}
 \end{equation}
 For any integer $Q > 1$ there is a positive constant $c_Q$ and a smooth function $h(x, y)$ defined on 
 $ (0, \infty) \times (-\infty, \infty)$ such that
 \begin{equation}
  \delta(n)= \frac{c_Q}{Q^2} \sum_{q=1}^\infty  \  \sideset{}{'} \sum_{a (\textrm{mod} \  q)} e\left( \frac{an}{q}\right) h\left(
  \frac{q}{Q}, \frac{n}{Q^2}\right).
 \end{equation} 
The constant $c_Q$ satisfies 
\[
 c_Q = 1+ O_N \left(Q^{-N} \right), 
\] for any $N>0$. Moreover $h(x, y) \ll x^{-1}$ for all $y$, and $h(x,y)$ is non-zero only when $x \leq \max \{1, 2|y| \} $. 

\noindent 
The smooth function $h(x,y)$ satisfies 
\begin{equation}
 x^{i} \frac{\partial h}{\partial x^{i}} (x, y) \ll_i x^{-1} \ \ \ \textrm{and} \ \ \ \ \frac{\partial h}{\partial y}= 0, 
\end{equation}
for $x\leq 1$ and $|y| \leq \frac{x}{2}$. And also for $|y|\geq \frac{x}{2}$, we have
\begin{equation} \label{i=0}
 x^i y^j \frac{\partial^{i+j} h (x, y)}{\partial  x^i \partial y^j}  \ll_{i, j } \frac{1}{x}. 
\end{equation}

\end{lemma}
\begin{proof}
 See  \cite{HB}. 
\end{proof}

\begin{lemma} \label{hb}
Let $N,$ $m$ and $n$ be non-negative integers.  Let $ x= o \left( \min \{1,  |y| \} \right)$. Then we have
\[
\frac{\partial^{m+n} h (x, y)}{\partial  x^m \partial y^n} \ll_{m, n, N} \frac{1}{x^{1+m+n}}  \left(x^N + \min \left\lbrace 1 , \left( \frac{x}{|y|} \right)^N \right\rbrace \right).
\]
The term $x^N$ on the right may be omitted for $n = 0$.
\end{lemma}

\begin{proof}
See \cite[Lemma 5]{HB}
\end{proof}

Next we recall the Poisson summation formula. 

\begin{lemma} \label{poisson}
 {\bf Poisson summation formula}: $f:\mathbb{R } \rightarrow \mathbb{R}$ is any Schwarz class function. Fourier transform 
 of $f$ is defined  as 
 \[
  \widehat{f}(y) = \int_{ \mathbb{R}} f( x) e(- x   y) dx,
 \] where $dx$ is the usual Lebesgue measure on $ \mathbb{R } $. 
 We have 
\begin{equation*}
 \sum_{ n \in \mathbb{Z}  }f(n) = \sum_{m \in \mathbb{Z} } \widehat{f}(m). 
\end{equation*} If $W(x)$ is any smooth and compactly supported function on $\mathbb{R}$, we have:
\begin{align*}
\sum_{n \in \mathbb{Z}  }e\left( \frac{an}{q}\right) W\left( \frac{n}{X}\right) = \frac{X}{q} \sum_{ m \in \mathbb{Z}  } \sum_{\alpha (\textrm{mod} q )}  e\left(\frac{\alpha + m}{q} \right) \widehat{W} \left( \frac{mX}{q} \right). 
\end{align*} 
\end{lemma}
 \begin{proof}
 See  \cite[page 69]{IK}.
 \end{proof}

\begin{remark} \label{remark}
If $W(x)$ satisfies $x^j W^{{j}} (x)\ll1$, then it can be easily shown, by integrating by parts  that dual sum is essentially supported on $ m\ll \frac{q (qX)^{\epsilon}}{X}$. The contribution coming from $m\gg \frac{q (qX)^{\epsilon}}{X} $ is negligibly small. 
\end{remark}
We shall recall the Voronoi summation formula for $SL(2, \mathbb{Z})$ automorphic forms. For sake of exposition we only present the case of Maass forms. The case of holomorphic forms is even simpler.
Let $ f $ be a Maass form with Laplacian eigenvalue $1/4+ \nu^2$ and with Fourier expansion
\[   
\sqrt{y} \sum_{n \neq 0 } \lambda(n) K_{i\nu}  (2\pi |n| y ) e(n x). 
\]
We shall use the following Voronoi type summation formula, which was first proved by Meruman \cite{MERU0}.
\begin{lemma} \label{voronoi}
{\bf Voronoi summation formula}:
Let $h$ be a compactly supported smooth function on the interval $(0, \infty)$. We have
\begin{equation} \label{varequation}
\sum_{n=1}^\infty \lambda (n) e_q(an) h(n) = \frac{1}{q} \sum_{\pm} \sum_{n=1}^\infty \lambda(\mp n) e_q(\pm \overline{a}n) H^{\pm} \left( \frac{n}{q^2}\right),
\end{equation}
where $ a \overline{a} \equiv 1 (\mod  q)$, and 
\begin{align*}
&H^{-} (y)= \frac{- \pi}{\cosh( \pi \nu)} \int_0^\infty h(x) \left\lbrace  Y_{2i\nu } + Y_{-2i\nu }\right\rbrace \left( 4\pi \sqrt{xy}\right) dx, \\
&H^{+} (y)= 4\cosh( \pi \nu) \int_0^\infty h(x)  K_{2i\nu }  \left( 4\pi \sqrt{xy}\right) dx,
\end{align*} where $Y_{2i\nu } $ and $ K_{2i\nu }$ are Bessel's functions of first and second kind and $e_q(x)= e^{\frac{2 \pi i x}{q}}$.

\end{lemma}
 \begin{remark}
 When $h$ is supported on the interval $[X, 2X]$ and satisfies $x^j h^{(j)} (x) \ll 1$, then  integrating by parts and using the properties of Bessel's function it is easy to see that the sums on the right hand side of equation \eqref{varequation} are essentially supported on $n\ll_{f, \epsilon} q^2 (qX)^\epsilon/X$.  
 For smaller values of $n$ we will use trivial bound that is $ H^{\pm} \left( \frac{n}{q^2}\right) \ll X$.  
 \end{remark}


\section { P\lowercase{roof of }  T\lowercase{heorem} 1.1 } 
 We shall prove the case when all $f_1, f_2$ and $f_3$ are Maass forms. The case of holomorphic eigenforms are similar (even relatively simple). We first substitute $ n+h = m (\sim N)$ and $\delta (m,n):= \delta(m-n)$ where $\delta(n)$ is defined by equation \eqref{del}. We  have
\begin{align*}
S &:=  \frac{1}{H}\sum_{h\leq H}  \sum_{m, n \in \mathbb{Z}}      \lambda_1 (n)  \lambda_2 (m) \lambda_3 (n+ 2h) \delta (m,n)  W_1\left( \frac{n}{N} \right) W_2 \left( \frac{m}{N} \right) V \left( \frac{h}{H} \right), 
\end{align*} 

where  $W_1(x)$,  $W_2(x)$ and $V(x)$ are smooth bump functions supported on the interval $[1,2 ]$ and satisfying 
\[
 x^{(j) } V^{(j)} (x),  \ \ x^{(j) } W_{\ell}^{(j)} (x) \ll 1, \ \ \ \textrm{for} \vspace{10pt} \ \ \ \ell = 1, 2\ \ \ \   \textrm{and} \ \ \   j\in \mathbb{Z}, j\geq 0. 
\] By  Lemma \ref{deltan} we write


\begin{align*}
S &= \frac{c_{Q_1} }{HQ_1^2} \sum_{h\leq H} \sum_{m, n \in \mathbb{Z}} \sum_{q_1\leq Q_1} \sideset{}{'} \sum_{a_1( q_1)}   \lambda_1 (n)  \lambda_2 (m) \lambda_3 (n+ 2h) e\left( \frac{a_1 (n+h -m)}{q_1}\right)   \\ 
& \hspace{60pt}  \times  h\left( \frac{q_1}{Q_1}, \frac{n+h -m}{Q_1^2}  \right) W_1 \left(\frac{n}{N} \right) W_2 \left(\frac{m}{N} \right)V \left( \frac{h}{H} \right).
\end{align*} 
 We substitute $n+ 2h = l ( \sim N)$. We  apply Lemma \ref{deltan} once again to obtain
\begin{align*}
S &= \frac{c_{Q_1} c_{Q_2}}{HQ_1^2 Q_2^2} \sum_{h\leq H} \sum_{m, n, l \in \mathbb{Z}} \sum_{q_1\leq Q_1} \sideset{}{'} \sum_{a_1( q_1)}  \sum_{q_2\leq Q_2} \sideset{}{'} \sum_{a_2( q_2)}  \lambda_1 (n)  \lambda_2 (m) \lambda_3 (l) e\left( \frac{a_1 (n+h -m)}{q_1}\right)  \\ 
&    \hspace{60pt} \times\left(  \frac{a_2 (n+2h-l)}{q_2}\right)h\left( \frac{q_1}{Q_1}, \frac{n+h -m}{Q_1^2}  \right) h\left( \frac{q_2}{Q_2}, \frac{n+2h -l}{Q_2^2}  \right)  W_1 \left(\frac{n}{N} \right)  \\
&    \hspace{160pt}  \times W_2 \left(\frac{m}{N} \right)   W_3 \left(\frac{l}{N} \right) V \left( \frac{h}{H} \right)\\
&= \frac{c_{Q_1} c_{Q_2}}{HQ_1^2 Q_2^2}  \sum_{q_1\leq Q_1} \sideset{}{'} \sum_{a_1( q_1)}  \sum_{q_2\leq Q_2} \sideset{}{'} \sum_{a_2( q_2)}  \left(\sum_{ n \in \mathbb{Z}} \lambda_1 (n)  e\left( \frac{a_1 q_2 + a_2 q_1 }{q_1 q_2} n \right)  W_1 \left(\frac{n}{N} \right)\right) \times \\
& \left(\sum_{m \in \mathbb{Z}} \lambda_2 (m)  e\left( -\frac{a_1 m }{q_1 }  \right) W_2 \left(\frac{m}{N} \right)  \right)  \left( \sum_{ l \in  \mathbb{Z}} \lambda_3 (l)e\left( -\frac{a_2 l }{q_2 }   \right) W_3 \left(\frac{l}{N} \right) \right)  \\
 &  \hspace{30pt}\times  \sum_{h\leq H} e\left( \frac{a_1 q_2 + 2a_2 q_1 }{q_1 q_2} h \right) h\left( \frac{q_1}{Q_1}, \frac{n+h -m}{Q_1^2}  \right) h\left( \frac{q_2}{Q_2}, \frac{n+2h -l}{Q_2^2}  \right)
 V \left( \frac{h}{H} \right).
\end{align*}  

\subsection{Applying Poisson summation formula}
We first write $ h= \alpha + b q_1 q_2$, and then apply the Poisson summation formula in   variable $b$. We set $Q_1= Q_2 =  \sqrt{X}$. We have 

\begin{align} \label{poisson1} 
&\sum_{h\in \mathbb{Z}} e\left( \frac{a_1 q_2 + 2a_2 q_1 }{q_1 q_2} h \right) \times
  h\left( \frac{q_1}{Q_1}, \frac{n+h -m}{Q_1^2}  \right) h\left( \frac{q_2}{Q_2}, \frac{n+2h -l}{Q_2^2}  \right)  V\left( \frac{h}{H} \right) \notag\\
 &  =  \sum_{\alpha (\textrm{mod}\ q)}  e\left( \frac{a_1 q_2 + 2a_2 q_1+h}{q_1 q_2} \alpha \right)
 \sum_{h\in \mathbb{Z} }\int_{ \mathbb{R}} h\left( \frac{q_1}{Q_1}, \frac{n+ \alpha +x q_1 q_2 -u}{Q_1^2}  \right)  \notag \\
   & \hspace{1cm} \times h\left( \frac{q_2}{Q_2}, \frac{n+2( \alpha +x q_1 q_2  ) -v}{Q_2^2} 
  \right) V\left( \frac{\alpha + x q_1 q_2}{H}\right)  e(-hx)dx \notag\\
 &=\frac{H}{q_1 q_2} 
 \sum_{h \in \mathbb{Z}}  \mathfrak{C} ( a_1, a_2, q_1, q_2) \mathfrak{J} (h;  n, m, l,  q_1, q_2), 
\end{align}
  after substituting $ \frac{\alpha + x q_1 q_2}{H} = y$, where the character sum $ \mathfrak{C} (a_1, a_2, q_1, q_2)$ is given by
 \begin{equation}  \label{charsum}
 \mathfrak{C} (a_1, a_2, q_1, q_2) = \sum_{\alpha (\textrm{mod}\ q)}  e\left( \frac{a_1 q_2 + 2a_2 q_1+h}{q_1 q_2} \alpha \right)
 \end{equation} and $ \mathfrak{J} (m, n , r, q)$ is given by
 
 \begin{align*}
 \mathfrak{J} (h; n, u, v,  q_1, q_2) & = \int_{ \mathbb{R}} h\left( \frac{q_1}{Q_1}, \frac{n+  x H -u}{Q_1^2}  \right) h\left( \frac{q_2}{Q_2}, \frac{n+2xH  -v}{Q_2^2} 
  \right) \\
   & \hspace{2cm} \times V\left( x\right)  e\left( - \frac{h Hx}{q_1 q_2}\right) dx.
 \end{align*}
Note that the $x$-intetral is supported only on the interval $[1, 2]$.  Applying  integration-by-parts  $j$ times and bounds of function $V(x)$ and $h(x,y )$ (listed in Lemma \ref{deltan} ) we  have 

\begin{align*}
\mathfrak{J} (h; n, u, v,  q_1, q_2) &= \int_{ \mathbb{R}}  \sum_{p_1 + p_2 +p_3 = j} h^{(p_1)}\left( \frac{q_1}{Q_1}, \frac{n+  x H -u}{Q_1^2}  \right) h^{(p_2)}\left( \frac{q_2}{Q_2}, \frac{n+2xH  -v}{Q_2^2} 
  \right) \\
   & \hspace{2cm} \times V^{(p_3)}\left( x\right)  e\left( - \frac{h Hx}{q_1 q_2}\right) \left( \frac{q_1 q_2}{h H}\right)^j dx\\
&\ll \left( \frac{q_1 q_2}{h H}\right)^j \left\lbrace  \sum_{p_1 + p_2 +p_3 = j} \left(1+ \left( \frac{H}{Q_1^2} \right)^{p_1}  \left( \frac{Q_1}{q_1} \right)^{p_1+1} +  \left( \frac{H}{Q_2^2} \right)^{p_2}  \left( \frac{Q_2}{q_2} \right)^{p_2+1}  \right) \right \rbrace \\
&\ll Q_1 Q_2 \left( \frac{q_1 q_2}{hH} + \frac{q_2}{h Q_1} + \frac{q_1}{h Q_2}\right)^j \ll  Q_1 Q_2 \left( \frac{q_1 q_2}{hH} +\frac{1}{h}\right)^j. 
\end{align*}

We note that if  $ h \gg  (Q X)^\epsilon \left( \frac{q_1 q_2 }{H}+1 \right)$, then contribution of $ \mathfrak{J} (h; n, m, l,  q_1, q_2) $ is negligibly small  i.e., of order $ O_A(X^{-A})$ for any $ A >0$. Evaluating the character sum given in equation \eqref{charsum}, we can write the right hand side of equation \eqref{poisson1} as
 
 \begin{align} \label{afterpoisson}
 = H \sum_{\substack{ |h|\leq \frac{q_1 q_2}{H}+1\\  a_1 q_2 + 2a_2 q_1+h \equiv 0 (q_1 q_2)}} \mathfrak{J} (h;  n, m, l,  q_1, q_2).  
 \end{align}
 
 \subsection{Applying Voronoi summation formula }

 We shall apply Voronoi summation formula simultaneously to  sum over $l$ and $m$,  where  compactly supported function $h(x)$  is replaced by  \\ $ W_2(u) W_3 (v)\mathfrak{J} (h; n, u, v ,  q_1, q_2)$. We have
\begin{align*}
& \sum_{m, l \in \mathbb{Z}} \lambda_2 (m)   \lambda_3 (l) e\left( -\frac{a_1 m }{q_1 }  \right) e\left( -\frac{a_2 l }{q_2 }   \right) W_2 \left(\frac{m}{N}   \right)W_3 \left(\frac{l}{N} \right)  \mathfrak{J} (h; n, m, l,  q_1, q_2) \\
&= \frac{1}{q_1 q_2}\sum_{m, l\ll \frac{Q^2 (QX)^\epsilon}{N}} \lambda_2 (m)  \lambda_3 (l) e\left( \frac{\overline{a_1} m }{q_1 }  \right) e\left(\frac{\overline{a_2} l }{q_2 }   \right) \mathcal{H}_{ h}^{\pm} \left( n;  \frac{m}{q_1^2},  \frac{l}{q_2^2}\right) + O_A\left(N^{-A} \right), 
\end{align*} 
where  
\begin{align*}
\mathcal{H}_h^{+} (w; y, z)&=  4 \cosh(\pi \nu) \iint_{\mathbb{R}^2} W_2\left(\frac{u}{N} \right) W_3 \left( \frac{v}{N} \right) \int_{ \mathbb{R}} h\left( \frac{q_1}{Q_1}, \frac{w + x H -u}{Q_1^2}  \right) h\left( \frac{q_2}{Q_2}, \frac{w + 2xH  -v}{Q_2^2} \right)   \\
  &   \hspace{2cm} \times V\left( x\right)  e\left( - \frac{h Hx}{q_1 q_2}\right) dx    K_{2 i\nu} \left(4 \pi \sqrt{y  u}  \right) K_{2 i\nu} \left(4 \pi \sqrt{z v}\right) du \ dv. 
\end{align*}
we have similar expression for $ \mathcal{H}_h^{-} (w; y, z)$, where Bessel function $ K_{2 i\nu} (x)$ is replaced  by Bessel function $\{ Y_{2 i\nu} + Y_{-2 i\nu}\} (x) $. 

We first make change of variables $\frac{u}{N} = u^\prime$ and $\frac{v}{N} = v^\prime$, we have

\begin{align*}
\mathcal{H}_h^{+} (w; y, z)&=  N^2 4 \cosh(\pi \nu) \iint_{1\leq u^\prime, v^\prime\leq 2} W_2(u^\prime) W_3  (v^\prime) \int_{ \mathbb{R}} h\left( \frac{q_1}{Q_1}, \frac{w + x H - Nu^\prime}{Q_1^2}  \right) h\left( \frac{q_2}{Q_2}, \frac{w + 2xH  -N v^\prime}{Q_2^2} \right)   \\
  &   \hspace{3cm} \times V\left( x\right)  e\left( - \frac{h Hx}{q_1 q_2}\right) dx    K_{2 i\nu} \left(4 \pi \sqrt{y  N u^\prime }  \right) K_{2 i\nu} \left(4 \pi \sqrt{z N v^\prime}\right) d u^\prime \ d v^\prime.
\end{align*} We have used the asymptotic formula for $ K_{2 i\nu} (x)$, $ Y_{2 i\nu}(x)$ and$  Y_{-2 i\nu} (x) $ to show that dual sum over $m$ and $l$ is supported  on $(QN)^\epsilon$, namely
\begin{equation} \label{bessel}
J_{k-1} (x), \ Y_{\pm2 i\nu}(x)  = e^{ix}U_{\pm2 i\nu}(x)  + e^{-ix} \overline{U}_{\pm2 i\nu}(x) \ \ \ \and \ \ \  \left|x^{k} K_\nu^{(k)} (x) \right| \ll_{k , \nu} \frac{ e^{-x} (1+ \log |x|)}{(1+x)^{  1/2 }},
\end{equation}

  where the function $ U_{\pm2 i\nu}(x)$ satisfies,
$$ x^j U_{\pm2 i\nu}^{(j)} (x) \ll_{j , \nu, k} (1+x)^{-1/2}.$$ Since $K$ Bessel function has exponential decay, hence integral is negligibly small for $ \frac{m}{Q^2} N \gg (QN)^\epsilon$. In case of $Y$ Bessel function, integrating by parts we have
\begin{align*}
\int_{\mathbb{R}} U_1(x) Y_{\pm2 i\nu} \left( 4 \pi \frac{\sqrt{nNx}}{q} \right) dx &= 2 \int_{\mathbb{R}} U_1(y^2) y Y_{\pm2 i\nu} \left( 4 \pi \frac{\sqrt{nN}}{q} y \right) dy \\
&= \int_{\mathbb{R}} U_1(y^2) y U_{\pm2 i\nu} \left( 2\frac{\sqrt{nN}}{q} y \right)  e\left( 2\frac{\sqrt{nN}}{q} y \right) dy \\
& \ll M \left( \frac{D q}{\sqrt{nN}} \right)^j,
\end{align*} where $U_1(x)$ is smooth bump function supported on the interval $[N, 2N]$ and satisfies $x^j U_1^{(j)} \ll D^j$. Hence integral is negligibly small if $ \frac{D q}{\sqrt{nN}} \ll 1$, that is $n\gg \frac{q^2 D^2}{N}$. In our case, from Lemma \ref{hb} we have $D = Q/q$. So after Voronoi summation formula, summation over $m$ and $l$ are supported on $m, l \ll \frac{Q^2 (QN)^\epsilon}{N}$. 

We want to calculate the derivative of $ \mathcal{H}_h^{+} (w; y, z)$ with respect to variable $w$.  We first substitute the following change of variables,

\noindent
$ w + xH -Nu^\prime = Nu  \hspace{10pt} \textrm{and} \hspace{10pt} w + 2xH -N v^\prime = N v^\prime. $
We have
\begin{align*}
\mathcal{H}_h^{+} (w; y, z)&=   N^2 4 \cosh(\pi \nu) \iint_{\mathbb{R}^2} W_2\left(\frac{w + xH - N u}{N} \right) W_3 \left( \frac{w + 2xH - N v}{N} \right) \int_{ \mathbb{R}} h\left( \frac{q_1}{Q_1}, \frac{N u}{Q_1^2}  \right)  \\
  &  \hspace{2cm}  \times  h\left( \frac{q_2}{Q_2}, \frac{N v}{Q_2^2} \right)  V\left( x\right)  e\left( - \frac{h Hx}{q_1 q_2}\right) dx     K_{2 i\nu} \left(4 \pi \sqrt{y  (w + xH - N u )}  \right)  \\
  &  \hspace{2cm} \times  K_{2 i\nu} \left(4 \pi \sqrt{z (w + 2 xH - N u )}\right) du \ dv.
\end{align*} For $z>0$, for all $\nu$ and $k\geq 0$, we have following  bound for Bessels functions (see \cite[Lemma C.1 and C.2]{ KMV} )
\begin{align*}
&\left( z^\nu J_\nu (z)\right)^\prime = z^\nu J_{\nu-1} (z),  \ \ \left( z^\nu K_\nu (z)\right)^\prime = z^\nu K_{\nu-1} (z), \ \ \  \left( z^\nu Y_\nu (z)\right)^\prime = z^\nu Y_{\nu-1} (z),\\
 \end{align*}
 Using Lemma \ref{hb} and above bounds for the Bessels functions we have:

\begin{align} \label{partial}
 &\frac{\partial}{\partial w}\mathcal{H}_h^{+} (w; y, z) \notag\\
 &=  \frac{\partial}{\partial w}  N^2 \frac{4 \cosh(\pi \nu)}{ Q_1^2} \iint_{  u\ll \frac{q_1}{Q_1}, v\ll \frac{q_2}{Q_2}} W_2\left(\frac{w + xH - N u}{N} \right)  \notag\\
 & \hspace{20pt} \times  W_3 \left( \frac{w + 2xH - N v}{N} \right) \int_{\mathbb{R}}  h\left( \frac{q_1}{Q_1}, \frac{N u}{Q_1^2}  \right) h\left( \frac{q_2}{Q_2}, \frac{N u}{Q_1^2}  \right) V\left( x\right)  e\left( - \frac{h Hx}{q_1 q_2}\right)  \notag\\
 &  \hspace{1cm}   \times  K_{2 i\nu} \left(4 \pi \sqrt{y  (w + xH - N u )}  \right) 
     K_{2 i\nu} \left(4 \pi \sqrt{z (w + 2 xH - N u )}\right)  du \ dv \ dx  \notag\\ 
  & \ll N^2 \frac{Q_1}{q_1} \frac{Q_2}{q_2} \iint_{  u\ll \frac{q_1}{Q_1}, v\ll \frac{q_2}{Q_2}} \left( \frac{1}{N} + \frac{1}{N}  \right. \qquad \notag\\
  &  \hspace{3cm}\left. + \frac{\sqrt{y}}{ \sqrt{w + xH - N u} } K_{2 i\nu}^\prime \left(4 \pi \sqrt{y  (w + xH - N u )}  \right) + \cdots \right) du \ dv \ dx  \notag\\
  & \ll N^2 \frac{Q_1}{q_1} \frac{Q_2}{q_2} \iint_{  u\ll \frac{q_1}{Q_1}, v\ll \frac{q_2}{Q_2}} \left( \frac{1}{N} + \frac{1}{N}  \right. \qquad \notag\\
  &  \hspace{3cm}\left. + \frac{1}{N}\sqrt{y(w + xH - N u)} K_{2 i\nu}^\prime \left(4 \pi \sqrt{y  (w + xH - N u )}  \right) + \cdots \right) du \ dv \ dx  \notag \notag \\
  &\ll N^2 \frac{Q_1}{q_1} \frac{Q_2}{q_2} \frac{1}{N} \frac{q_1}{Q_1} \frac{q_2}{Q_2} \ll N, 
\end{align} 
 as $x$-integral   is supported on the interval $[1, 2]$.

\noindent
From  the congruence relation given in the equation \eqref{afterpoisson},  we have
\[
a_1 q_2 + 2a_2 q_1+h \equiv 0 (q_1 q_2) \Rightarrow a_1 q_2 +h \equiv 0 (q_1 ) \ \  \textrm{ and} \ \   2a_2 q_1+h \equiv 0 ( q_2). 
\] 
 After the application of Poisson and Voronoi summation formulae,  we have:
 \begin{align} \label{lasts}
S &= \frac{c_{Q_1} c_{Q_2}}{Q_1^2 Q_2^2}  \sum_{q_1\leq Q_1}   \sum_{q_2\leq Q_2} \sum_{m, l \ll X^\epsilon}
\sum_{ h\leq \left( \frac{q_1 q_2}{H} + 1 \right)} \frac{1}{q_1 q_2} \lambda_2 (m)  \lambda_3 (l) e\left( \frac{\overline{a_1} m }{q_1 }  \right) e\left(\frac{\overline{a_2} l }{q_2 }   \right)  \times \notag\\ 
& \hspace{2cm} \sum_{ n \in \mathbb{Z}} \lambda_1 (n)  e\left( \frac{-h - a_2 q_1 }{q_1 q_2} n \right)  W_1 \left(\frac{n}{N} \right) \mathcal{H}_{ h}^{\pm} \left( n;  \frac{m}{q_1^2},  \frac{l}{q_2^2}\right) 
 \end{align}
 Now we use  Riemann Stieltjes integral to evaluate last sum in above equation.  For any $ \alpha \in \mathbb{R}$, by using cancellation in additive twist (see equations \eqref{sl2})  and equation \eqref{partial},  we have
 
 \begin{align} \label{lastsum} 
 \sum_{ n \in \mathbb{Z}} \lambda_1 (n)  e\left(  n \alpha \right )  W_1 \left(\frac{n}{N} \right) \mathcal{H}_{ h}^{\pm} \left( n;  \frac{m}{q_1^2},  \frac{l}{q_2^2}\right) &= \int_N^{2N}  \left( \sum_{ n \leq y} \lambda_1 (n)  e(  n \alpha) \right) \frac{\partial}{\partial y}\mathcal{H}_{ h}^{\pm} \left( y;  \frac{m}{q_1^2},  \frac{l}{q_2^2}\right) dy \notag \\
 &\ll \int_N^{2N} y^{1/2} N   dy \ll N^{5/2}.  
 \end{align}
 Substituting bound of equation \eqref{lastsum} in equation \eqref{lasts} we have
 
 \begin{align*}
 S &\ll \frac{c_{Q_1} c_{Q_2}}{Q_1^2 Q_2^2}  \sum_{q_1\leq Q_1}   \sum_{q_2\leq Q_2} \sum_{m, l \ll X^\epsilon}
\sum_{ h\leq \left( \frac{q_1 q_2}{H} + 1 \right)} \frac{1}{q_1 q_2}  N^{5/2}  \\
& \ll \frac{c_{Q_1} c_{Q_2} N^{5/2+\epsilon} }{Q_1^2 Q_2^2 }  \sum_{q_1\leq Q_1}   \sum_{q_2\leq Q_2} \frac{1}{q_1 q_2}  \left( \frac{q_1 q_2}{H} +1 \right)\ll \frac{ N^{3/2+\epsilon} }{ H} +   N^{1/2 +\epsilon} \ll \frac{ N^{3/2+\epsilon} }{ H}. 
 \end{align*}
as $ Q_1$ and $ Q_2 = \sqrt{N}$.

%

\vspace{1cm}

\noindent
{\bf Acknowledgement:} Author would like to thank Prof. Ritabrata Munshi for suggesting the problem and for all fruitful discussions and suggestions. Author would also like to thank Prof. Tim Browning and Vinay Kumaraswamy for their thorough analysis of the manuscripts and his suggestion about many slips in earlier version of this paper. 
{}

\vskip 1mm
\noindent 
\begin{thebibliography}{} 


\bibitem{BB}
 Baier, S.; Browning, T.D.; Marasingha, G.; and Zhao, L. : \emph{Averages of shifted convolutions of $d_3(n)$}, Proceedings of the Edinburgh Mathematical Society {\bf 55}, (2012), 551-576.

\bibitem{VB}
 Blomer, V. : \emph{On triple correlations of divisor functions}, http://arxiv.org/abs/1512.03278.


\bibitem{TD3}
 Browning, T.D. : \emph{The divisor problem for binary cubic form}, Journal de Th{\' e}orte des Nombres de Bordeaux {\bf 23}, 2011, 579-602. 
 
 


 
\bibitem{DS}
Daniel, S. : \emph{On the divisor-sum problem for binary quadratic forms}, J. reine angew Math. {\bf 507} 1999, 107-129.  
 
 
\bibitem{DFI} 
Duke, W., Frieldande, J.B. and Iwaniec, H. : \emph{A quadratic divisor problem}, Invent. Math. {\bf 115} (1994) 209-217. 
 
\bibitem{ET}
Estermann, T : \emph{{\"U}ber die Darstellung einer Zahl als Differenz von swei Produkten}, J. Reine Angew. Math. { \bf 164}, (1931), 173-182.  
 
 
 
 
 
\bibitem{FI}
Friedlander, John B, and Iwaniec, H. : \emph{Summation Formulae for the Coefficients of $L$-functions}, Canad. J. Math. Vol. {\bf 57}(3), 2005, pp. 494-505. 

%



\bibitem{HB} 
Heath-Brown, D.R. : \emph{A new form of circle method and  
its Application to quadratic forms}, J. Reine Angew. Math. {\bf 481}(1996), 149-206. MR 1421949 (97k:11139)

\bibitem{IN}
Ingham, A.E. : \emph{Some asymptotic formulae in the theory of numbers}, J. London Math. Soc. {\bf 2} (1927), 202-208. 

\bibitem{IV}
Ivi{\'c}, Aleksandar and Wu, Jie. : \emph{On the general additive divisor problem}, Proceedings of the
Steklov Institute of Mathematics, {\bf 276,} 2012,  pp.140-148. 

\bibitem{HI}
Iwaniec, H. : \emph{Topics in Classical Automorphic Forms}, Graduate text in mathematics {\bf 17}, American Mathematical Society, Providence, RI, 1997.

\bibitem{IK}
Iwaniec, H. and Kowalski, E. : \emph{Analytic Number Theory}, American Mathematical Society Colloquium Publication {\bf 53}, American Mathematical Society, Providence, RI, 2004. 

\bibitem{KS}
Kim, H. and Sarnak, P : \emph{Refined estimates towards the Ramanujan and Selberg conjectures}, J. American Math. Soc. {\bf 16}, (2003), 175-181.

\bibitem{KMV}
Kowalski, E; Michel, P. and   Vanderkam, J. : \emph{Rankin Selberg $L$ functions in level the aspect}, 

\bibitem{LIN}
Lin, Yongxiao : \emph{Triple correlation of Fourier coefficents of cusp forms},     


 https://arxiv.org/pdf/1607.02956.pdf. 


\bibitem{MERU0}

Meurman, T. : \emph{On exponential sums involving the Fourier coefficients of Maass wave forms}, J. Reine Angew. Math. {\bf 384} (1988), 192-207. 


\bibitem{MERU}

Meurman, T. : \emph{On the binary additive divisor problem}, Number theory (Turku, 1999), 223-246, de Gruyter, 2001.

%




\end{thebibliography}
\end{document}